\documentclass[]{interact}

\usepackage{xcolor}
\usepackage{array}

\usepackage{epstopdf}
\usepackage[caption=false]{subfig}

\usepackage[numbers,sort&compress]{natbib}
\bibpunct[, ]{[}{]}{,}{n}{,}{,}
\renewcommand\bibfont{\fontsize{10}{12}\selectfont}
\makeatletter
\def\NAT@def@citea{\def\@citea{\NAT@separator}}
\makeatother

\theoremstyle{plain}
\newtheorem{theorem}{Theorem}[section]
\newtheorem{lemma}[theorem]{Lemma}

\newtheorem{proposition}[theorem]{Proposition}
\newtheorem{clm}[theorem]{Claim}

\theoremstyle{definition}
\newtheorem{definition}[theorem]{Definition}

\theoremstyle{remark}

\newcommand{\ot}{\Omega_T }

\newcommand{\po}{\partial\Omega}
\newcommand{\mdiv}{\textup{div}}

\newcommand{\woto}{W^{1,2}\left(\Omega\right)}

\newcommand{\oa}{\Omega_a}
\newcommand{\oc}{\Omega_c}

\newcommand{\op}{\Omega^{\prime}}

\newcommand{\ifa}{i_{\mbox{\tiny{fara}}}}

\newcommand{\ifae}{i^{(\varepsilon)}_{\mbox{\tiny{fara}}}}
\newcommand{\qe}{Q^{(\varepsilon)}}

\newcommand{\pst}{\Phi_s^{(\tau)}}
\newcommand{\pet}{\Phi_e^{(\tau)}}
\newcommand{\te}{\theta_{\varepsilon}}
\newcommand{\psn}{\Phi_s^n}
\newcommand{\pen}{\Phi_e^n}

\begin{document}

	
	\title{An Existence Theorem for a Model of Temperature Within a Lithium-Ion Battery}
	
	\author{
		\name{B. C. Price\textsuperscript{a}\thanks{B. C. Price, Email: bcp193@msstate.edu} and Xiangsheng Xu\textsuperscript{b}\thanks{X. Xu, Email: xxu@math.msstate.edu}}
		\affil{Department of Mathematics and Statistics, Mississippi State
			University, Mississippi State, MS 39762.}
	}
	
	\maketitle
	
	\begin{abstract}
		In this article we investigate a model for the temperature within a Lithium-Ion battery. The model takes the form of a parabolic PDE for the temperature coupled with two elliptic PDE's for the electric potential within the solid and electrolyte phases. The primary difficulty comes from the coupling term, which is given by the Butler-Volmer equation. It features an exponential nonlinearity of both the electric potentials and the reciprocal of the temperature. Another difficulty arising in the temperature equation are the gradients of the electric potentials squared showing up on the right-hand side. Due to the nonlinearity, meaningful estimates for the temperature are currently not known. In spite of this, our investigation reveals the local existence of continuous temperature for the Lithium-Ion Battery. 
	\end{abstract}
	

	\section{Introduction}
	
	In this article we investigate a system of PDE's which model the temperature in a lithium-Ion Battery (or Li-battery for short).  Li-batteries have become a staple in modern technologies, appearing everywhere from our phones to our cars. Mathematical models of Li-batteries provide a way to manage such batteries in a safe and reliable way \cite{CKCAK}. While there are many studies describing Lithium-Ion batteries, (see for instance \cite{CKCAK, CK, WGL, GW, SW, TND, AMR, DCR, TW}), the majority of results are concerning the P2D model and are one dimensional. The only exceptions that we are aware of are the models studied in \cite{WXZ} and \cite{X.Xuo}, and neither of those considered the effects of temperature. In this article we aim to understand the effects of temperature in a Li-Ion battery in multiple space dimensions. 
	
	We now proceed to briefly describe the model. A lithium-ion cell is dividing into three regions; a porous positive electrode, a separator, and a negative electrode. We will denote the lithium-ion cell itself as $\Omega$, then the positive and negative electrodes we denote by $\Omega_c$ and $\Omega_a$ respectively. Finally, we will use $\Omega_s$ to represent the separator.
	
	\begin{figure}[h]
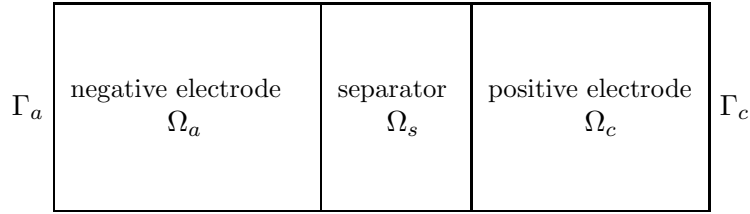

		\centering
			$\Gamma_a$ \begin{tabular}{ | m{8em} | m{4em}| m{7em} | }
				\hline
				& & \\
				& & \\
				{\small negative electrode}&  {\small separator}&{\small positive electrode}\\
				\hspace{.45in}
				$\Omega_a$ &   \hspace{.2in}  $\Omega_s$ &
				\hspace{.45in}
				$\Omega_c$  \\
				& & \\
				& &  \\
				\hline
			\end{tabular} $\Gamma_c$
			\hspace{-.5in} \caption{ The domain $\Omega$}
			\label{fig2}
		\end{figure}
		The negative electrode is typically made of carbon while the positive electrode is made of a metal oxide, such as lithium cobalt oxide or lithium iron phosphate. They are modeled as an agglomeration of spherical solid particles. The entire cell $\Omega$ is filled with an electrolyte solution, which 
		is a lithium salt in an organic solvent. When the battery discharges, the lithium ions move from the negative electrode through the electrolyte to the positive electrode where they combine with electrons from the external circuit to form an ionized metal oxide. During charging, an external electrical current is applied to reverse the process and move the lithium ions back to the negative electrode. 
		The electrode is viewed as a superposition of active material, filler, and electrolyte, and these phases coexist at every point in the model.

		\subsection{Derivation of the Model Equations}
		The state variables required to describe the model are as follows:
		\begin{eqnarray*}
			u(x,t)&=& \mbox{the temperature, }\\
			\Phi_s(x,t)&=&\mbox{the electric potential  in the solid electrode,}\\
			\mathbf{i}_s(x,t)&=&\mbox{ the current density  in the solid electrode},\\
			j_n(x,t)&=&\mbox{the molar flux of lithium at the surface of the spherical particle},\\ 
			\Phi_e(x,t)&=&\mbox{the electric potential  in the  electrolyte,}\\
			\mathbf{i}_e(x,t)&=&\mbox{ the current density  in the  electrolyte},\\
			C_e(x,t)&=&\mbox{ the concentration of lithium ions in the  electrolyte},\\
			C_s(x,r,t)&=&\mbox{ the concentration of lithium ions at a distance $r$ from the center of}\nonumber\\
			&&	\mbox{ a spherical particle located at $x$ in the solid electrode .}
		\end{eqnarray*}
		The input to the model is the external current density $I$ applied to the battery. Kirchoff's law asserts
		$$I=\mathbf{i}_e+\mathbf{i}_s.$$
		On the other hand, at each point in the electrode, the net pore-wall molar flux is related to the divergence of the current via the formula
		\begin{equation}
			\mdiv \mathbf{i}_e=A_sFj_n\equiv A_si_{\mbox{\tiny{fara}}} \ \ \mbox{ in $\op\equiv\Omega_a\cup\Omega_c$, } \nonumber 
		\end{equation}
		where $i_{\mbox{\tiny{fara}}} $ denotes the charge transfer density.
		Next, we look to Ohm's law, which tells us that the current is proportional to the electric field. The electric field can be obtained as the negative gradient of the electric potential. Thus, looking at the solid phase electric potential, Ohm's law tells us, 
		\begin{equation}
			-\sigma_s \nabla \Phi_s = \mathbf{i}_s = I - \mathbf{i}_e \ \ \mbox{ in $\op$. } \nonumber 
		\end{equation}
		We may now take the divergence of this equation to find,
		\begin{equation}\label{s.p.equ}
			-\mdiv\left(\sigma_s\nabla\Phi_s\right) =-A_si_{\mbox{\tiny{fara}}} \ \ \mbox{ in $\op$. } 
		\end{equation}
		
		For the electric potential in the electrolyte we need to use the electrochemical version of Ohm's law, which accounts for the combined effects of concentration and electric potential, \cite{WGL}, and we can write it as,
		\begin{equation}
			-\sigma_e\nabla \Phi_e +\frac{2Ru\sigma_e}{F} \left(1-t^0_+\right) \nabla \ln C_e = \mathbf{i}_e. \nonumber 
		\end{equation}
		Where $K$ is the concentration dependent conductivity, and $\alpha$ is a positive constant. Upon taking the divergence of this equation we obtain,
		\begin{equation}
			-\mdiv\left(\sigma_e\nabla \Phi_e -\frac{2Ru\sigma_e}{F} \left(1-t^0_+\right) \nabla \ln C_e \right) =A_si_{\mbox{\tiny{fara}}}\chi_{\op} \ \ \mbox{in $\Omega$}. \nonumber 
		\end{equation} 
		To prescribe boundary conditions,  we introduce the following notations as in \cite{X.Xuo}
		\begin{align}
			\Gamma_a&=\po\cap\partial\oa  \nonumber \\
			\Gamma_c&=\po\cap\partial\oc. \nonumber 
		\end{align}
		Then, the external boundary of $\op$ is
		$$\partial_{ext}\op\equiv\Gamma_a\cup\Gamma_c.$$
		Then the equations \eqref{s.p.equ} and \eqref{e.p.equ} are coupled  with the boundary conditions, 
		\begin{align}
			\frac{\partial \Phi_e}{\partial \mathbf{n}} &= 0 \ \ \ \mbox{on $\partial\Omega$,} \label{I.el.Pot.Boundary} \\
			\frac{\partial \Phi_s}{\partial \mathbf{n}} &= 0 \ \ \ \mbox{on $\partial\Omega^{\prime}\setminus\left(\Gamma_a\cup\Gamma_c\right)$,} \label{I.s.Pot.int.Bound.} \\
			\sigma\frac{\partial \Phi_s}{\partial \mathbf{n}} &= I \ \ \ \mbox{on $\Gamma_a\cup\Gamma_c$} \label{I.s.Pot.Ext.Bound.}
		\end{align}
		Next, we look at the concentration of lithium-ions in the electrolyte $C_e$. The concentration $C_e$ changes due to the flux, $j$, of lithium-ions into and out of the solid particles. Thus, the conservation of mass can be written as, \cite{JDL}
		\begin{equation}\label{masscon.}
			\varepsilon_e \frac{\partial C_e}{\partial t}+ \mdiv J = \frac{A_s}{F}i_{\mbox{\tiny{fara}}}\chi_{\op}\ \ \mbox{in $\ot$}. 
		\end{equation}
		Where $J$ is given by Fick's Law, $$J = -D_e\nabla C_e+\frac{t^0_+}{F}\mathbf{i}_e  $$ and $D_e$ is the diffusion coefficient.
		Finally, for the concentration of lithium-ions in the solid particles, conservation of mass will once again give us a diffusion equation. We will assume that particles have no interaction with neighboring particles. The only portion of the solid phase concentrations that speak to the rest of the model is the surface concentration of each particle. Diffusion being the dominating phenomenon \cite{CKCAK} means that we may simplify a bit and consider a radially symmetric diffusion equation,
		\begin{equation}
			\partial_t C_s -\frac{1}{r^2} \frac{\partial}{\partial r}\left(r^{2}D_s\frac{\partial C_s}{\partial r} \right) = 0.
		\end{equation}
		The flux of lithium-ions flowing into the cell can then be captured in the boundary conditions,
		\begin{equation}
			\frac{\partial C_s}{\partial r}(x,R(x),t) = -\frac{1}{D_sF}i_{\mbox{\tiny{fara}}}, \ \ \frac{\partial C_s}{\partial r}(x, 0,t)=0.
		\end{equation}
		There is no flow of lithium-ions at the center of each particle which gives us the second boundary condition,
		Conservation of energy asserts that the temperature $T$ satisfies,
		\begin{eqnarray}
			(\rho C_p)\partial_tu-\mdiv\left(k\nabla u\right)&=&Q\ \ \mbox{in $\ot$},\\
			-k\nabla u\cdot\mathbf{n}&=&k_1(u-T_a)\ \ \mbox{on $\Sigma_T$,}\\
			u(x,0)&=&u_0(x)\ \ \mbox{on $\Omega$.}
		\end{eqnarray}
		Where,
		\begin{eqnarray}
			Q&=&\left(A_si_{\mbox{\tiny{fara}}}\eta+\sigma_s|\nabla\Phi_s|^2\right)\chi_{\op}\nonumber\\
			&&+\sigma_e|\nabla\Phi_e|^2+\frac{2RT\sigma_e}{F}(1-t^0_+)\nabla\ln C_e\cdot\nabla\Phi_e.
		\end{eqnarray}
		Then, the charge transfer density $i_{\mbox{\tiny{fara}}}$ is related to the potentials $\Phi_s$, $\Phi_e$ and the temperature $u$ through the Butler-Volmer equation,
		\begin{eqnarray}\label{ifaradefn}
			i_{\mbox{\tiny{fara}}}&=&i_{0,\mbox{\tiny{fara}}}\left[\exp\left(\frac{\alpha F}{Ru}\eta\right)-\exp\left(-\frac{\alpha F}{Ru}\eta\right)\right],\\
			i_{0,\mbox{\tiny{fara}}}&=&Fk_0C_e^{\alpha}\left(C_s^{\tiny{\mbox{Max}}}-C_s^{\tiny{\mbox{Surf}}}\right)^{\alpha}\left(C_s^{\mbox{\tiny{Surf}}}\right)^{\alpha},\\
			\eta&=&\Phi_s-\Phi_e-U
			,\\
		\end{eqnarray}
		Collecting each of these together, and assuming a given function to represent the concentrations we arrive at the following system of PDE's,
		\begin{align}
			(\rho C_p)\partial_tu-\mdiv\left(k\nabla u\right)&=Q\ \ \mbox{in $\ot$}, \label{temp} \\
			-\mdiv\left(\sigma_s\nabla\Phi_s\right) &=-A_si_{\mbox{\tiny{fara}}} \ \ \mbox{ in $\op_T$. }  \label{solidpot} \\
			-\mdiv\left(\sigma_e\nabla \Phi_e \right)&=-\mdiv\left(d_1\sigma_e u\bf{f}\right)  +A_si_{\mbox{\tiny{fara}}}\chi_{\op} \ \ \mbox{in $\ot$}. \label{electrolytepot}
		\end{align}
		Where, 
		\begin{align}
			i_{\mbox{\tiny{fara}}} &= g_1\left[e^{\frac{\alpha(\Phi_s-\Phi_e) }{u}}e^{\frac{-\alpha U}{u}}-e^{-\frac{\alpha (\Phi_s-\Phi_e)}{u}}e^{\frac{\alpha U}{u}}\right], \nonumber \\
			Q&=\left(A_si_{\mbox{\tiny{fara}}}(\Phi_s - \Phi_e)+\sigma_s|\nabla\Phi_s|^2\right)\chi_{\op}\nonumber\\
			&+\sigma_e|\nabla\Phi_e|^2+d_1\sigma_e u \bf{f}\cdot\nabla\Phi_e. \nonumber 
		\end{align}

		\subsection{Main Results}
		Before stating our Main Theorem let us lay out the assumptions on our data, 
		\begin{itemize}
			\item (H1) $\rho$, $C_p$, $T_a$, $\alpha$, $A_s$, $k_1$and $d_1$ are positive constants,
			\item (H2) The functions $k$, $\sigma_s$, and $\sigma_e$ are each bounded, positive functions which are also bounded away from zero,
			\item (H3) The function $U$ is a bounded continuous function on $\op_T$, 
			\item (H4) $g_1$ is a bounded continuous function on $\ot$, and there exists a constant $g_0$ so that, $g_1 \geq g_0 >0$,
			\item (H5) $\mathbf{f}$ is a bounded vector field so that $\mathbf{f}\cdot\mathbf{n} = 0$ on $\partial\Omega$, 
			\item (H6) $I$ is a bounded function and, 
			\begin{equation}
				\int_{\partial\op} I dS = 0, \nonumber 
			\end{equation}
			\item (H7) We assume that, $\left|\Omega_s\right| < \frac{1}{2} \left|\Omega_a\right|$, so that the separator is strictly smaller than the rest of the battery
		\end{itemize}

		We define the spaces, 
		\begin{align}
			&\mathbb{X} = L^{\infty}\left(\op_T\right)\times L^{\infty}\left(\ot\right), \nonumber  \\
			&\mathbb{Y} = L^p\left(0,T;W^{1,p}\left(\op\right)\right)\times L^p\left(0,T;W^{1,p}\left(\Omega\right)\right). \nonumber 
		\end{align}
		For more detailed definitions of the spaces we use throughout this work see the definitions given in the preliminary section.

		Next we give our definition of a weak solution to \eqref{temp}-\eqref{electrolytepot}. 
		\begin{definition}\label{weaksolutiondefn}
			We say that a triple $\left(u,\Phi_s, \Phi_e\right)$ is a weak solution to \eqref{temp}-\eqref{electrolytepot} if the following conditions hold:
			\begin{itemize}
				\item (D1) The function $u$ satisfies, $u>0$, $u\in C\left(\overline{\ot}\right)\cap L^2\left(0,T;\woto\right)$, 
				\item (D2) The functions $\left(\Phi_s,\Phi_e\right)$ satisfy $\left(\Phi_s,\Phi_e\right) \in \mathbb{X}\cap\mathbb{Y}$, and, 
				\begin{equation}
					\int_{\Omega} \Phi_e dx + \int_{\op} \Phi_s dx = 0, \ \ \mbox{ a.e. on $\left(0,T\right)$. } \nonumber 
				\end{equation}
				\item (D3) The functions $\left(u, \Phi_s, \Phi_e \right)$ satisfy the integral equations, 
				\begin{align}
					-\int_{\ot} u \partial_t \varphi_1 dxdt + \int_{\ot} k\nabla u \cdot \nabla \varphi_1 dxdt &= \int_{\ot} Q\left(u,\Phi_s-\Phi_e, \nabla\Phi_s, \nabla\Phi_e\right)\varphi_1 dxdt \nonumber \\
					&+ \int_{\Omega} u_0(x)\varphi_1(x,0) dx \nonumber \\
					\int_{\op_T} \sigma_s \nabla\Phi_s\cdot \nabla \varphi_2 dx dt &= -A_s\int_{\op_T} \ifa\left(u,\Phi_s-\Phi_e\right) \varphi_2 dxdt \nonumber \\
					 & + \int_{0}^{T} \int_{\partial\op} I \varphi_2 dxdt \nonumber \\
					\int_{\ot} \sigma_e\nabla\Phi_e\cdot\nabla\varphi_3 dxdt &= d_1\int_{\ot} \sigma_e u \mathbf{f} \cdot \nabla\varphi_3 dxdt \nonumber \\
					& + A_s \int_{\op_T} \ifa\left(u,\Phi_s-\Phi_e\right)\varphi_3 dxdt \nonumber 
				\end{align}
				for all $\left(\varphi_1,\varphi_2,\varphi_3\right)\in \left(H^1\left(0,T;W^{1,2}\left(\Omega\right)\right)\right)^3$ so that $\varphi_i\left(x,T\right)=0$ for $i=1,2,3$.
			\end{itemize}
		\end{definition}
		
		Our main Theorem is the following; 
		\begin{theorem}\label{MainTheorem}
			Under the assumptions (H1)-(HN), there is a number $T^* > 0$ so that equations \eqref{temp}-\eqref{electrolytepot} have a weak solution in the sense of definition \ref{weaksolutiondefn} in $\Omega_{T^*}$. 
		\end{theorem}
		
		Before continuing we would like to make a few remarks about our result. The first is that while Theorem \ref{MainTheorem} guarantees a small time $T^*$ for which the temperature remains bounded, we are still not able to determine the behavior of $u$ has $T\rightarrow \infty$. The terms $\left|\nabla\phi_s\right|^2$ and $\left|\nabla\Phi_e\right|^2$ appearing in $Q$ are related to $e^{\frac{1}{u}}$ through the Butler-Volmer equation, and to $u^2$ through \eqref{electrolytepot}. Parabolic equations with a squared term are already known to feature blow up, \cite{CF}, and the addition of the exponential nonlinearity in the Butler-Volmer equation is not expected to improve the situation. 
		\newline
		In order to avoid the singularity at $u=0$ in the Butler-Volmer equation, and the possibility of blow-up for $u$, we first replace $u$ in \eqref{electrolytepot} and in \eqref{ifaradefn} with the expression, 
		\begin{equation}
			u \rightarrow u_0 - \te\left(u_0 - u\right) \nonumber 
		\end{equation}
		Where $\te$ is given as in \eqref{tedefn}, and $\varepsilon$ is chosen to be smaller than the minimum value of $u_0$. The removes the possibility of a singularity coming from $u$, and allows us to prove the existence of continuous solutions to the modified problem for all time. The continuity of solutions then allow us to choose a small time $T^*$ so that, 
		\begin{equation}
			\left|u_0 - u\right| \leq \varepsilon \ \mbox{ in $\Omega_{T^*}$. } \nonumber 
		\end{equation}
		Then, over the time interval $\left[0,T^*\right]$ we have that, 
		\begin{equation}
			u_0-\te\left(u_0-u\right) = u \nonumber 
		\end{equation}
		This provides us with a continuous solution to \eqref{temp}-\eqref{electrolytepot}.

		The layout for the rest of the article is as follows. In the second section we will detail the basic preliminary definitions and results which we will use throughout the rest of the work. In the third section we will fix $u$ to be a given continuous function.  We then prove the existence of bounded weak solutions $\Phi_s$ and $\Phi_e$. If $\Phi_s$ and $\Phi_e$ are bounded, then they must belong to $W^{1,p}\left(\Omega\right)$ on account of lemma \ref{lpelliptic}. This in turn will imply that $u$ must be continuous from the classical regularity theory for parabolic equations \cite{LSU}. Then, in the fourth and final section, we will use the results of section three to design a fixed-point map over the continuous functions for the temperature $u$. We use the Leray-Schauder Theorem \ref{LSTheorem} to prove the existence of a fixed-point and obtain theorem \ref{MainTheorem}.   
		
		\section{Preliminaries}
		
		In this section we will lay out the relevant definitions and results which will be used throughout this work. We begin by defining the function spaces which we will be working in.  
		
		We first define the Sobolev spaces. We suppose that $\Omega$ is a bounded domain in $\mathbb{R}^N$. For a given $p\geq1$ and a positive integer $k$ we set,
		\begin{equation}
			\left\|u\right\|_{W^{k,p}\left(\Omega\right)} = \left(\int_{\Omega} \sum_{\left|\alpha\right|\leq k} \left|D^{\alpha} u\right|^p dx \right)^{\frac{1}{p}} \nonumber 
		\end{equation} 
		We then define the Sobolev space $W^{k,p}\left(\Omega\right)$ to be the closure of $C^{\infty}\left(\Omega\right)$ under the norm $\left\|\cdot\right\|_{W^{k,p}\left(\Omega\right)}$. Properties of the spaces $W^{k,p}\left(\Omega\right)$ can be found in \cite{GT} chap. 7.
		
		We will also make use of the spaces of H\"{o}lder continuous functions. For $\beta\in\left(0,1\right)$ We define the space of H\"{o}lder continuous functions with exponent $\beta$, $C^{\beta}\left(\overline{\Omega}\right)$, to be all of those functions $f$ so that,
		\begin{equation}
			\left[f\right]_{\beta,\Omega} = \sup_{x,y\in\Omega} \frac{\left|f(x)-f(y)\right|}{\left|x-y\right|^{\beta}} < \infty. \nonumber 
		\end{equation}
		In a similar way we define the parabolic H\"{o}lder space $C^{\beta,\frac{\beta}{2}}\left(\overline{\ot}\right)$ as all of the functions $f$ so that,
		\begin{equation}
			\left[f\right]_{\beta,\ot} = \sup_{\left(x,t\right),\left(y,\tau\right)\in\ot } \frac{\left|f(x,t)-f(y,\tau)\right|}{\left(\left|x-y\right|^2+\left|t-\tau\right|\right)^{\frac{\beta}{2}} } < \infty. \nonumber 
		\end{equation}
		More information on these spaces can be found in \cite{GT},\cite{LSU}. 
		
		Our existence assertion is provided by the Leray-Schauder fixed-point Theorem, (\cite{GT} chap. 11)
		\begin{lemma}\label{LSTheorem}
			Let $\mathbb{T}$ be a compact mapping of a Banach space $\mathbb{B}$ into itself, and suppose there exists a constant $M$ such that,
			\begin{equation}
				\left\| x \right\|_{\mathbb{B}} < M, \nonumber 
			\end{equation}
			for all $x \in \mathbb{B}$ and $\sigma \in \left[0,1\right]$ satisfying $x = \sigma \mathbb{T}x$. Then $\mathbb{T}$ has a fixed point. 
		\end{lemma}

		Next we collect some known estimates for parabolic and elliptic equations that we will need for our later development. Our first is an $L^p$ estimate for elliptic equations in divergence form. \cite{HD}
		\begin{lemma}\label{lpelliptic}
			Suppose that $p\in\left(1,\infty\right)$ and $u$ is a weak solution in $W^{1,2}\left(\Omega\right)$ to the elliptic boundary-value problem, 
			\begin{align}
				-\mdiv\left( A\nabla u \right) &= -\mdiv\left(\mathbf{g} \right) + f, \ \mbox{in $\Omega$ } \nonumber \\
				(A\nabla u-\mathbf{g}) \cdot \mathbf{n} &= h, \ \mbox{ on $\partial\Omega$ } \nonumber 
			\end{align}
			Then there is a constant $c$ so that, 
			\begin{align}
				\left\|u \right\|_{W^{1,p}\left(\Omega\right)} \leq c \left\| f\right\|_{p,\Omega} + c \left\| g\right\|_{p,\Omega} + c\left\|h\right\|_{p,\Omega} \nonumber 
			\end{align}
		\end{lemma}
		
		Our next lemma is a well-known $L^{\infty}$ estimate for parabolic Neumann problems. The proof is well known see for instance \cite{LSU}. 
		\begin{lemma}\label{linfparabolic}
			Suppose that $u$ is a weak solutions of the initial boundary-value problem, 
			\begin{align}
				\partial_t u - \mdiv\left(k \nabla u \right) &= f, \ \mbox{ in $\ot$, } \nonumber \\
				k\nabla u\cdot \mathbf{n} &= g, \ \mbox{ on $ \Sigma_T$, } \nonumber \\
				u\left(x,0\right) &= u_0\left(x\right), \ \mbox{ on $\Omega\times\left(0,T\right)$. } \nonumber 
			\end{align}
			Let $p> \frac{N}{2} + 1$. Then, 
			\begin{equation}
				\sup_{\ot} \left|u \right| \leq c\left\|f\right\|_{p,\ot} + c\left\|g\right\|_{\infty,\ot} + 2\left\|u_0\right\|_{\infty,\Omega} + 1. \nonumber 
			\end{equation}
		\end{lemma}

		\section{Existence Results for the Electric Potentials}
		
		In this section we prove some existence results for the electric potentials $\Phi_s$ and $\Phi_e$. For the duration of this section we will assume that the temperature $u$ is a given continuous function on $\ot$. 
		
		To get started, we first define the function $\te$ on $\mathbb{R}$ by, 
		\begin{equation}\label{tedefn}
			\te\left(s\right) = \left\{ \begin{array}{lr} \varepsilon &  s > \varepsilon \\
				s & \left|s\right| \leq \varepsilon, \\
				-\varepsilon & s<-\varepsilon, \end{array} \right.
		\end{equation}
		
		Then let, 
		\begin{equation}\label{epsiloncondition}
			L_0 = \min_{\Omega} u_0 > 0, \ \ \varepsilon \in (0,L_0)
		\end{equation}
			For $\varepsilon \in (0, L0)$, we have,
		\begin{equation}
			u_0 - \te\left(u_0 - u \right) \geq L_0 - \varepsilon > 0. 
		\end{equation}
		We then approximate $i_{fara}$ as, 
		\begin{equation}\label{ifaraepsilon}
			\ifae(u, \Phi_s-\Phi_e) = g_1\left[e^{\frac{\alpha(\Phi_s-\Phi_e)}{u_0 -\te\left(u_0-u\right)}} e^{\frac{-\alpha U}{u_0-\te\left(u_0-u\right)} } - e^{\frac{-\alpha\left(\Phi_s-\Phi_e\right)}{u_0-\te\left(u_0-u\right)}} e^{\frac{\alpha U}{u_0 -\te\left(u_0-u\right)} } \right] 
		\end{equation}
		By our choice of $\varepsilon$, the singularity coming from the term $\frac{1}{u}$ has been taken care of.

		Next, as a first step towards proving our main theorem we consider the following problem, 
		\begin{align}
			-\mdiv\left(\sigma_s\nabla\Phi_s\right) +\tau \Phi_s &= -A_s\ifae\left(u, \Phi_s-\Phi_e\right), \ \ \mbox{ in $\op$, } \label{psapprox} \\
			-\mdiv\left(\sigma_e\nabla \Phi_e\right) + \tau \Phi_e &= -\mdiv\left(d_1\sigma_e \left( u_0 - \te\left(u_0-u\right)\right) \mathbf{f} \right) + A_s \ifae\left( u, \Phi_e - \Phi_s\right)\chi_{\op}, \ \  \mbox{ in $\Omega$, } \label{peapprox} \\ 
			\sigma_s \nabla \Phi_s \cdot \mathbf{n} &= I, \ \ \mbox{ on $\partial\op$, } \label{psapproxboundary} \\
			\sigma_e \nabla \Phi_e \cdot \mathbf{n} &= 0 , \ \ \mbox{ on $\partial\Omega$. } \label{peapproxboundary} 
		\end{align}
		Set
		$$\mathbb{Y}_1=W^{1,2}(\op)\times W^{1,2}(\Omega).$$
		For a given function $u \in C\left(\overline{\ot} \right)$, and given number $\tau > 0$. A weak solution to \eqref{psapprox}-\eqref{peapproxboundary} are functions $\left(\Phi_s, \Phi_e\right)$ in the spaces, $\mathbb{X}\times\mathbb{Y}_1$ such that, 
		\begin{align}
			\int_{\op} \sigma_s \nabla \Phi_s \cdot \nabla \varphi_1 dx + \tau \int_{\op} \Phi_s \varphi_1 dx &= - A_s \int_{\op} \ifae\left(u, \Phi_s-\Phi_e\right) \varphi_1 dx + \int_{\partial\op} I \varphi_1 dS, \nonumber \\
			\int_{\Omega} \sigma_e \nabla \Phi_e \cdot \nabla \varphi_2 dx + \tau \int_{\op} \Phi_e \varphi_2 dx &= d_1\int_{\Omega} \sigma_e \left( u_0 - \te\left(u_0-u\right)\right) \mathbf{f}\cdot \nabla \varphi_2 dx \nonumber \\
			& + A_s \int_{\op} \ifae\left(u,\Phi_s-\Phi_e\right)\varphi_2 dx \nonumber 
		\end{align}
		For all $\left(\varphi_1,\varphi_2\right) \in \left(W^{1,2}\left(\Omega\right)\right)^2$. 
		
		We now have the following existence result for \eqref{psapprox}-\eqref{peapproxboundary}, 
		\begin{proposition}\label{potapprox}
			Let $u \in C\left(\overline{\ot}\right)$ and $\tau > 0$. Then, for every $t\in\left[0,T\right]$ there is a unique weak solution to \eqref{psapprox}-\eqref{peapproxboundary}. 
		\end{proposition}
		\begin{proof}
			We first show that any solution $\left(\Phi_s, \Phi_e\right)$ must be unique. To do so, we first suppose that $\left(\Phi^{(1)}_s, \Phi^{(1)}_e\right), \left(\Phi^{(2)}_s, \Phi^{(2)}_e\right) \in \mathbb{X} \times \mathbb{Y}_1$ are two solutions to \eqref{psapprox}-\eqref{peapproxboundary}. Then, by subtracting the equations for $\left(\Phi^{(2)}_s, \Phi^{(2)}_e\right)$ from those for $\left(\Phi^{(1)}_s, \Phi^{(1)}_e\right)$ we derive, 
			\begin{align}
				-&\mdiv\left[\sigma_s\nabla\left(\Phi^{(1)}_s-\Phi^{(2)}_s\right)\right] + \tau\left(\Phi^{(1)}_s-\Phi^{(2)}_s\right) \nonumber \\
				&= -A_s\ifae\left(u,\Phi^{(1)}_s-\Phi^{(1)}_e\right) + A_s\ifae\left(u,\Phi^{(2)}_s-\Phi^{(2)}_e\right), \ \ \mbox{ in $\op$, } \label{psdifference} \\
				-&\mdiv\left[\sigma_e\nabla\left(\Phi^{(1)}_e - \Phi^{(2)}_e\right)\right] + \tau\left(\Phi^{(1)}_e-\Phi^{(2)}_e\right) \nonumber \\
				&= A_s\ifae\left(u,\Phi^{(1)}_s-\Phi^{(1)}_e\right) -A_s \ifae\left(u, \Phi^{(2)}_s-\Phi^{(2)}_e\right)\chi_{\op}, \ \ \mbox{ in $\Omega$, } \label{pedifference} \\
				&\sigma_s \nabla\left(\Phi^{(1)}_s-\Phi^{(2)}_s\right) \cdot \mathbf{n} = 0, \ \ \mbox{ on $\partial\op$, }  \\
				&\sigma_e \nabla \left(\Phi^{(1)}_e-\Phi^{(2)}_e\right) \cdot \mathbf{n} = 0, \ \ \mbox{ on $\partial\Omega$. } 
			\end{align}
			We now use $\left(\Phi^{(1)}_s - \Phi^{(2)}_s\right)$ as a test function in \eqref{psdifference}, and $\left(\Phi^{(1)}_e - \Phi^{(2)}_e\right)$ as a test function in \eqref{pedifference} and add together the resulting equations to derive, 
			\begin{align}
				\int_{\op}& \sigma_s \left|\nabla\left(\Phi^{(1)}_s - \Phi^{(2)}_s\right)\right|^2 dx + \int_{\Omega} \sigma_e \left|\nabla\left(\Phi^{(1)}_e - \Phi^{(2)}_e\right)\right|^2 dx \nonumber  \\
				& + \tau \int_{\op} \left[ \left(\Phi^{(1)}_s - \Phi^{(2)}_s\right) \right]^2 dx + \tau \int_{\Omega} \left[ \left(\Phi^{(1)}_e - \Phi^{(2)}_e\right) \right]^2 dx \nonumber  \\
				& = - A_s \int_{\op} G dx. \label{uniquenessone}
			\end{align}
			Where, 
			\begin{equation}
				G = \left[ \ifae\left(u, \Phi^{(1)}_s-\Phi^{(1)}_e\right) - \ifae\left(u, \Phi^{(2)}_s-\Phi^{(2)}_e\right)\right]\left[ \left(\Phi^{(1)}_s-\Phi^{(1)}_e\right)-\left(\Phi^{(2)}_s-\Phi^{(2)}_e\right)\right] \nonumber 
			\end{equation}
			For the integral of $G$ we use the fact that $\ifae\left(y_1, y_2\right)$ is an increasing function of its second variable to conclude, 
			\begin{equation}
				G\geq 0. \nonumber 
			\end{equation}
			Thus the right-hand side of \eqref{uniquenessone} is non-positive. As a result, equation \eqref{uniquenessone} implies that, 
			\begin{equation}
				\Phi^{(1)}_s = \Phi^{(2)}_s, \quad \Phi^{(1)}_e = \Phi^{(2)}_e, \ \ \mbox{ a.e. on $\op$ and $\Omega$ respectively. } \nonumber
			\end{equation}
			This completes the proof of uniqueness.

			We now move to the proof of existence. To prove existence we will use the Leray-Schauder Theorem \ref{LSTheorem}. We define a mapping $\mathbb{B}$, taking $C\left(\overline{\op}\right) \times C\left(\overline{\Omega}\right)$ into itself, in the following manner, given $\left(\varphi_s, \varphi_e\right) \in C\left(\overline{\op}\right) \times C\left(\overline{\Omega}\right)$, we define $\left(\Phi_s, \Phi_e\right)$ to be the solutions to the problems, 
			\begin{align}
				-\mdiv\left(\sigma_s\nabla\Phi_s\right) + \tau\Phi_s &= -A_s\ifae\left(u,\varphi_s-\varphi_e\right), \ \ \mbox{ in $\op$, } \label{psexists} \\
				-\mdiv\left(\sigma_e\nabla\Phi_e\right) + \tau\Phi_e &= -\mdiv\left(d_s\sigma_e\left(u_0-\te\left(u_0 - u\right)\right) \mathbf{f}\right) + A_s\ifae\left(u,\varphi_s-\varphi_e\right)\chi_{\op}, \ \ \mbox{ in $\Omega$, } \label{peexists} \\
				\sigma_s \nabla\Phi_s \cdot \mathbf{n} &= I, \ \ \mbox{ on $\partial\op$, } \\
				\sigma_e \nabla \Phi_e \cdot \mathbf{n} &= 0, \ \ \mbox{ on $\partial\Omega$. } \label{peboundaryexists} 
			\end{align}
			By the classical existence and regularity theory for linear elliptic equations, (\cite{GT}, chap. 8), we can conclude that there are unique weak solutions $\left(\Phi_s, \Phi_e\right) \in \mathbb{X}_1 \cap \mathbb{Y}_1\equiv C^{0,\beta}(\overline{\op})\times C^{0,\beta}(\overline{\Omega})\cap W^{1,2}(\op)\times W^{1,2}(\Omega)$ to equations \eqref{psexists}-\eqref{peboundaryexists} for some $\beta\in (0,1)$. We then put $\mathbb{B}\left(\varphi_s, \varphi_e\right) = \left(\Phi_s, \Phi_e\right)$. Clearly then, $\mathbb{B}$ is a well-defined operator on $C\left(\overline{\op}\right)\times C\left(\overline{\Omega}\right)$. Since $\mathbb{X}_1$ is compactly embedded in $C\left(\overline{\op}\right)\times C\left(\overline{\Omega}\right)$, we may also conclude that $\mathbb{B}$ maps bounded sets into precompact ones. 
			
			We now have the following claim, 
			\begin{clm}
				The mapping $\mathbb{B}$ is continuous on $C\left(\overline{\op}\right) \times C\left(\overline{\Omega}\right)$. 
			\end{clm}
			\begin{proof}
				Let $\{\left(\varphi^{(n)}_s, \varphi^{(n)}_e\right)\}$ be a convergent sequence in $C\left(\overline{\op}\right) \times C\left(\overline{\Omega}\right)$, and for each $n=1,2,3,...$ put $\left(\Phi^{(n)}_s, \Phi^{(n)}_e\right) = \mathbb{B}\left(\varphi^{(n)}_s,\varphi^{(n)}_e\right)$. This corresponds to the equations, 
				\begin{align}
					-\mdiv\left(\sigma_s\nabla\Phi^{(n)}_s\right) + \tau\Phi^{(n)}_s &= -A_s\ifae\left(u,\varphi^{(n)}_s-\varphi^{(n)}_e\right), \ \ \mbox{ in $\op$, } \label{c1}  \\
					-\mdiv\left(\sigma_e\nabla\Phi^{(n)}_e\right) + \tau\Phi^{(n)}_e &= -\mdiv\left(d_s\sigma_e\left(u_0-\te\left(u_0 - u\right)\right) \mathbf{f}\right) \nonumber \\
					& + A_s\ifae\left(u,\varphi^{(n)}_s-\varphi^{(n)}_e\right)\chi_{\op}, \ \ \mbox{ in $\Omega$, } \label{c2} \\
					\sigma_s \nabla\Phi^{(n)}_s \cdot \mathbf{n} &= I, \ \ \mbox{ on $\partial\op$, } \label{c3} \\
					\sigma_e \nabla \Phi^{(n)}_e \cdot \mathbf{n} &= 0, \ \ \mbox{ on $\partial\Omega$. } \label{c4} 
				\end{align}
				Since the sequence $\left(\varphi^{(n)}_s, \varphi^{(n)}_e\right)$ is uniformly bounded in $ C\left(\overline{\op}\right) \times C\left(\overline{\Omega} \right)$ we can conclude that, $\ifae\left(u, \varphi^{(n)}_s-\varphi^{(n)}_e\right)$ is uniformly bounded in $L^{\infty}\left(\op\right)$. As a result, we then have that $\left(\Phi^{(n)}_s, \Phi^{(n)}_e\right)$ is uniformly bounded in $\mathbb{X}_1\cap \mathbb{Y}_1$. Hence, we can pass to a subsequence, which we won't relabel, so that, 
				\begin{align}
					\Phi^{(n)}_s &\rightarrow \Phi_s, \ \ \mbox{ in $C\left(\overline{\op}\right)$ and weakly in $W^{1,2}\left(\op\right)$, } \nonumber \\
					\Phi^{(n)}_e &\rightarrow \Phi_e, \ \ \mbox{ in $C\left(\overline{\Omega}\right)$ and weakly in $W^{1,2}\left(\Omega\right)$. } \nonumber 
				\end{align}
				Then, we can pass to the limit in equations, \eqref{c1}-\eqref{c4} to get, 
				\begin{align}
					-\mdiv\left(\sigma_s\nabla\Phi_s\right) + \tau\Phi_s &= -A_s\ifae\left(u,\varphi_s-\varphi_e\right), \ \ \mbox{ in $\op$, } \label{c5}  \\
					-\mdiv\left(\sigma_e\nabla\Phi_e\right) + \tau\Phi_e &= -\mdiv\left(d_s\sigma_e\left(u_0-\te\left(u_0 - u\right)\right) \mathbf{f}\right) + A_s\ifae\left(u,\varphi_s-\varphi_e\right)\chi_{\op}, \ \ \mbox{ in $\Omega$, } \label{c6} \\
					\sigma_s \nabla\Phi_s \cdot \mathbf{n} &= I, \ \ \mbox{ on $\partial\op$, } \label{c7} \\
					\sigma_e \nabla \Phi_e \cdot \mathbf{n} &= 0, \ \ \mbox{ on $\partial\Omega$. } \label{c8} 
				\end{align}
				We then subtract \eqref{c5} and \eqref{c6} from \eqref{c1} and \eqref{c2} respectively to obtain, 
				\begin{align}
					-&\mdiv\left(\sigma_s\nabla\left(\Phi^{(n)}_s-\Phi_s\right)\right) + \tau\left(\Phi^{(n)}_s-\Phi_s\right) \nonumber \\
					&= -A_s\left(\ifae\left(u,\varphi^{(n)}_s-\varphi^{(n)}_e\right)-\ifae\left(u,\varphi_s-\varphi_e\right)\right), \ \ \mbox{ in $\op$, } \label{c9}   \\
					-&\mdiv\left(\sigma_e\nabla\left(\Phi^{(n)}_e-\Phi_e\right)\right) + \tau\left(\Phi^{(n)}_e-\Phi_e\right) \nonumber \\
					&=A_s\left(\ifae\left(u,\varphi^{(n)}_s-\varphi^{(n)}_e\right)-\ifae\left(u,\varphi_s-\varphi_e\right)\right)\chi_{\op}, \ \ \mbox{ in $\Omega$, } \label{c10} \\
					&\sigma_s \nabla\left(\Phi^{(n)}_s-\Phi_s\right) \cdot \mathbf{n} = 0, \ \ \mbox{ on $\partial\op$, }  \\
					&\sigma_e \nabla\left( \Phi^{(n)}_e-\Phi_e\right) \cdot \mathbf{n} = 0, \ \ \mbox{ on $\partial\Omega$. } 
				\end{align}
				Upon using $\left(\Phi^{(n)}_s-\Phi_s\right)$ as a test function in \eqref{c9} and $\left(\Phi^{(n)}_e-\Phi_e\right)$ as a test function in \eqref{c10} we derive, 
				\begin{align}
					\Phi^{(n)}_s &\rightarrow \Phi_s, \ \ \mbox{ in $C\left(\overline{\op}\right)$ and strongly in $W^{1,2}\left(\op\right)$, } \nonumber \\
					\Phi^{(n)}_e &\rightarrow \Phi_e, \ \ \mbox{ in $C\left(\overline{\Omega}\right)$ and strongly in $W^{1,2}\left(\Omega\right)$. } \nonumber
				\end{align}
				Finally, since solutions to \eqref{c5}-\eqref{c8} are unique, we can conclude that every subsequence of $\{\left(\Phi^{(n)}_s,\Phi^{(n)}_e\right)\}$ has a further convergent subsequence, all of which converge to the same limit $\mathbb{B}\left(\varphi_s,\varphi_e\right)$. Therefore, the overall limit converges, and we can conclude that $\mathbb{B}$ is continuous. 
			\end{proof}
			
			In order to use the Leray-Schauder Theorem, we still need to show that if $\delta\in\left(0, 1\right)$ and $\left(\Phi_s, \Phi_e\right) \in C\left(\overline{\op}\right) \times C\left(\overline{\Omega}\right)$ is such that, 
			\begin{equation}\label{deltaequation}
				\left(\Phi_s, \Phi_e\right) = \delta\mathbb{B}\left(\Phi_s, \Phi_e\right),
			\end{equation}
			Then there is a constant $c$ so that,
			\begin{equation}\label{LScond}
				\left\|\left(\Phi_s\chi_{\op},\Phi_e\right)\right\|_{\infty,\Omega} \leq c. 
			\end{equation}
			Our next two claims will deal with this. In order to use the results for our later development we will make the constant $c$ to be independent of $\tau$. Since the terms with $u$ are locked inside of the function $\te$, this constant will also be independent of $u$. 
			
			First notice that the equation \eqref{deltaequation} is equivalent to the system of equations, 
			\begin{align}
				-\mdiv\left(\sigma_s\nabla\Phi_s\right) + \tau\Phi_s &= -\delta A_s \ifae\left(u,\Phi_s-\Phi_e\right), \ \ \mbox{ on $\op$, } \label{deltaeqone} \\
				-\mdiv\left(\sigma_e\nabla\Phi_e\right) + \tau\Phi_e &= -\delta \mdiv\left(\sigma_e d_1\left(u_0-\te\left(u_0-u\right)\right)f\right) \nonumber \\
				& + \delta A_s \ifae\left(u,\Phi_s-\Phi_e\right)\chi_{\op}, \ \ \mbox{ on $\Omega$, } \label{deltaeqtwo} \\
				\sigma_s\nabla\Phi_s \cdot \mathbf{n} &= \delta I, \ \ \mbox{ on $\partial\op$, } \label{deltaeqthree} \\
				\sigma_e \nabla \Phi_e \cdot \mathbf{n} &=0, \ \ \mbox{ on $\partial\Omega$, } \label{deltaeqfour}
			\end{align}
			Then, we have the following two estimates, 
			\begin{clm}\label{potLtwobound}
				There is a constant $c$, which is independent of both $\tau$ and $u$, so that, 
				\begin{equation}
					\int_{\Omega} \left|\Phi_e\right|^2 dx + \int_{\op} \left|\Phi_s\right|^2 dx \leq c. \nonumber
				\end{equation}
			\end{clm}
			\begin{proof}
				To begin we first integrate \eqref{deltaeqone} over $\op$ and then integrate \eqref{deltaeqtwo} over $\Omega$, and then add together the two resulting equations to find,
				\begin{equation}\label{potaverageiszero}
					\int_{\Omega} \Phi_e dx + \int_{\op} \Phi_s dx = 0. 
				\end{equation}
				
				Then we use $\Phi_s$ as a test function in \eqref{deltaeqone} and $\Phi_e$ as a test function in \eqref{deltaeqtwo} and add together the two resulting equations to find, 
				\begin{align}\label{lto}
					\int_{\op} &\sigma_s\left|\nabla\Phi_s\right|^2 dx + \int_{\Omega} \sigma_e\left|\nabla\Phi_e\right|^2 dx + \tau\int_{\op} \Phi_s^2 dx + \tau\int_{\Omega} \Phi_e^2 dx \nonumber \\
					& + \delta A_s \int_{\op} \ifae\left(u,\Phi_s-\Phi_e\right)\left(\Phi_s-\Phi_e\right)dx \nonumber \\
					&= \delta\int_{\Omega} \sigma_e d_1\left(u_0-\te\left(u_0-u\right)\right) \mathbf{f}\cdot \nabla \Phi_e dx + \delta \int_{\partial\op} I \Phi_s dx. 
				\end{align}
				For the two integrals on the right-hand side we estimate, 
				\begin{align}\label{ltt}
					\delta\int_{\Omega} \sigma_e d_1\left(u_0-\te\left(u_0-u\right)\right) \mathbf{f}\cdot \nabla \Phi_e dx &\leq \delta c\left(\left\|u_0\right\|_{\infty,\Omega} + \varepsilon\right) \left\| f\right\|_{\infty,\Omega} \int_{\Omega} \left|\nabla \Phi_e\right| dx \nonumber \\
					& \leq \frac{1}{2} \int_{\Omega} \left|\nabla\Phi_e\right|^2 dx + \delta c\left(\left\|u_0\right\|_{\infty,\Omega} + \varepsilon\right)^2 \left\| f\right\|^2_{\infty,\Omega}
				\end{align}
				\begin{align}\label{ltth}
					\delta\int_{\partial\op} I \Phi_s dS &\leq \delta I_{max} \int_{\partial\op} \left|\Phi_s\right| dS 
				\end{align}
				Then, for the third integral on the left-hand side of equation \eqref{lto} we first note by the Mean Value Theorem that, 
				\begin{equation}
					\ifae\left(u,\Phi_s-\Phi_e\right) - \ifae\left(u,0\right) = \partial_{y_2} \ifae\left(u,\xi\right)\left(\Phi_s-\Phi_e\right). \nonumber 
				\end{equation}
				Then, from \eqref{ifaraepsilon} and (H4), we have, 
				\begin{align}\label{partialy2}
					\partial_{y_2}\ifae\left(u,\xi\right) &= \frac{g_1 \alpha}{u_0 - \te\left(u_0-u\right)} e^{\frac{\alpha\xi}{u_0-\te\left(u_0-u\right)} } e^{\frac{-\alpha U}{u_0-\te\left(u_0-u\right)} } \nonumber \\
					& + \frac{g_1 \alpha}{u_0 - \te\left(u_0-u\right)}e^{\frac{-\alpha\xi}{u_0-\te\left(u_0-u\right)} } e^{\frac{\alpha U}{u_0-\te\left(u_0-u\right)} }  \nonumber \\
					& \geq \frac{2 g_0 \alpha}{L_0 - \varepsilon} \nonumber \\
					& = c_0 \nonumber \\
					& > 0.
				\end{align}
				Subsequently we are able to estimate, 
				\begin{align}\label{ltf}
					\delta A_s& \int_{\op} \ifae\left(u,\Phi_s-\Phi_e\right) \left(\Phi_s-\Phi_e\right) dx \nonumber \\
					& = \delta A_s \int_{\op} \left[\ifae\left(u,\Phi_s-\Phi_e\right)-\ifae\left(u,0\right)\right]\left(\Phi_s-\Phi_e\right) dx \nonumber \\
					& + \delta A_s \int_{\op} \ifae\left(u,0\right)\left(\Phi_s-\Phi_e\right) dx \nonumber \\
					& = \delta A_s \int_{\op} \partial_{y_2} \ifae\left(u,\xi\right)\left(\Phi_s-\Phi_e\right)^2 dx + \delta A_s \int_{\op} \ifae\left(u,0\right)\left(\Phi_s-\Phi_e\right) dx \nonumber \\
					& \geq c_0 \delta A_s \int_{\op} \left(\Phi_s - \Phi_e\right)^2 dx - \frac{c_0 \delta A_s}{2} \int_{\op} \left(\Phi_s-\Phi_e\right)^2 dx - \delta c \int_{\op} \left|\ifae\left(u,0\right)\right|^2 dx 
				\end{align}
				Where in the last line we have also used Young's inequality. Upon substituting \eqref{ltt}, \eqref{ltth}, and \eqref{ltf} into inequality \eqref{lto} we obtain, 
				\begin{align}\label{ltfi}
					\int_{\op} &\sigma_s\left|\nabla\Phi_s\right|^2 dx + \int_{\Omega} \sigma_e\left|\nabla\Phi_e\right|^2 dx + \frac{\delta c_0 A_s}{2} \int_{\op} \left(\Phi_s-\Phi_e\right)^2 dx \nonumber \\
					& \leq \delta c \int_{\op} \left|\Phi_s\right| dx + \delta c \int_{\op} \left|\ifae\left(u,0\right)\right|^2 dx + \delta c \nonumber \\
					& \leq \delta c \int_{\op} \left|\Phi_s\right| dx + \delta c
				\end{align}
				From \eqref{ltfi} we are then able to derive, 
				\begin{equation}\label{ltdifference}
					\int_{\op} \left(\Phi_s-\Phi_e\right)^2 dx \leq c\int_{\op}\left|\Phi_s\right|dx + c. 
				\end{equation}
				Next, we note from \eqref{potaverageiszero} that, 
				\begin{align}\label{sos}
					\left(\int_{\Omega} \Phi_e dx \right)^2  + \left(\int_{\op} \Phi_s dx \right)^2 &= \frac{1}{2} \left( \left(\int_{\Omega} \left(\Phi_e  + \Phi_s\chi_{\op}\right) dx\right)^2 + \left(\int_{\Omega} \left(\Phi_e - \Phi_s\chi_{\op}\right) dx\right)^2 \right) \nonumber \\
					&= \frac{1}{2} \left(\int_{\op} \left(\Phi_e-\Phi_s\right) dx\right)^2 + \frac{1}{2}\left(\int_{\Omega\setminus\op} \Phi_e dx \right)^2. 
				\end{align}
				Then, using Poincare's inequality, \eqref{ltdifference}, and \eqref{sos} we have, 
				\begin{align}\label{PI}
					&\int_{\Omega} \Phi_e^2 dx + \int_{\op} \Phi_s^2 dx \nonumber \\
					& \leq 2 \int_{\Omega} \left(\Phi_e -\frac{1}{\left|\Omega\right|} \int_{\Omega} \Phi_e dx \right)^2 dx + \frac{2}{\left|\Omega\right|} \left(\int_{\Omega} \Phi_e dx \right)^2  \nonumber \\
					& + 2 \int_{\op} \left(\Phi_s -\frac{1}{\left|\op\right|} \int_{\op} \Phi_s dx \right)^2 dx + \frac{2}{\left|\op\right|} \left(\int_{\op} \Phi_s dx \right)^2 \nonumber \\
					& \leq c\int_{\Omega} \left|\nabla \Phi_e\right|^2 dx + c\int_{\op} \left|\nabla \Phi_s\right|^2 dx + \frac{2}{\left|\op\right|} \left( \left(\int_{\Omega} \Phi_e dx \right)^2  + \left(\int_{\op} \Phi_s dx \right)^2 \right) \nonumber \\
					& \leq c \int_{\op}\left|\Phi_s\right| dx + \frac{1}{\left|\op\right|} \left(\int_{\op} \left(\Phi_e-\Phi_s\right) dx \right)^2 + \frac{1}{\left|\op\right|} \left(\int_{\Omega\setminus\op} \Phi_e dx \right)^2 + c \nonumber \\
					& \leq c \int_{\op} \left|\Phi_s\right| dx + \frac{1}{\left|\op\right|} \left(\int_{\Omega\setminus\op} \Phi_e dx \right)^2 + c \nonumber \\
					& \leq \frac{1}{2} \int_{\op} \Phi^2_s dx + \frac{\left|\Omega_s\right|}{\left|\Omega_a\right|} \int_{\Omega} \Phi_e^2 dx + c.
				\end{align}
				Finally, due to (H7) the proof of the claim is complete. 
			\end{proof}
			With claim \ref{potLtwobound} out of the way we can now prove the following, 
			\begin{lemma}\label{potlinfbound}
				There is a constant $c$, which is independent of both $\tau$ and $u$, so that, 
				\begin{equation}
					\sup_{\Omega} \left|\Phi_e\right| + \sup_{\op} \left|\Phi_s\right| \leq c. \nonumber 
				\end{equation}
			\end{lemma}
			For the proof we will use a Moser style iteration argument. We are able to accomplish this despite the exponential nonlinearity by taking advantage of the fact that $\ifae$ is increasing in its second variable.  
			\begin{proof}
				We first let, 
				\begin{equation}
					b = \left\|\ifae\left(u,0\right)\right\|_{p, \op} + \left\|u_0\right\|_{\infty,\Omega} + I_{max} + \varepsilon. \nonumber 
				\end{equation}
				Then, for any $n\geq 1$, we use $\left(\Phi_s^{+} + b\right)^n $ as a test function in \eqref{deltaeqone} to arrive at, 
				\begin{align}\label{linfone}
					n &\int_{\op} \sigma_s \left(\Phi_s^{+} + b\right)^{n-1} \left|\nabla \Phi_s^+\right|^2 dx + \tau \int_{\op} \Phi_s \left(\Phi_s^{+} + b\right)^n dx \nonumber \\
					&= -\delta A_s \int_{\op} \ifae\left(u,\Phi_s-\Phi_e\right)\left(\Phi_s^{+} + b\right)^n dx + \delta\int_{\partial\op} I\left(\Phi_s^{+} + b\right)^n dS. 
				\end{align}
				Next we use $\left(\Phi_e^{+} + b\right)^n $ as a test function in \eqref{deltaeqtwo}, 
				\begin{align}\label{linftwo}
					n& \int_{\Omega} \sigma_e \left(\Phi_e^+ + b\right)^{n-1} \left|\nabla \Phi^{+}_e\right|^2 dx + \tau \int_{\Omega} \Phi_e \left(\Phi_e^+ + b \right)^n dx \nonumber \\
					& = \delta n \int_{\Omega} \sigma_e d_1\left(u_0 - \te\left(u_0 - u\right) \right) \left(\Phi_e^+ +b\right)^{n-1} f \cdot \nabla \Phi_e^{+} dx \nonumber \\
					& + \delta A_s \int_{\op} \ifae\left(u,\Phi_s-\Phi_e\right)\left(\Phi_e^+ + b\right)^n dx 
				\end{align}
				Next, we integrate \eqref{deltaeqone} over $\op$ to obtain, 
				\begin{equation}
					\tau \int_{\op} \Phi_s dx = - \delta A_s \int_{\op} \ifae\left(u,\Phi_s-\Phi_e\right) dx \nonumber 
				\end{equation}
				Upon separating $\Phi_s$ into its positive and negative portions we then find, 
				\begin{equation}\label{pshelp}
					\tau \int_{\op} \Phi_s^- dx = \tau \int_{\op} \Phi_s^+ dx + \delta A_s \int_{\op} \ifae\left(u,\Phi_s-\Phi_e\right) dx 
				\end{equation}
				Then using \eqref{pshelp} we calculate, 
				\begin{align}\label{pstau}
					\tau \int_{\op} \Phi_s\left(\Phi_s^+ + b\right)^n dx &= \tau \int_{\op} \Phi_s^+\left(\Phi_s^+ + b\right)^n dx - \tau \int_{\op} \Phi_s^- \left(\Phi_s^+ + b\right)^n dx \nonumber \\
					&=\tau \int_{\op} \Phi_s^+\left(\Phi_s^+ + b\right)^n dx - \tau b^n \int_{\op} \Phi_s^- dx \nonumber \\
					&= \tau \int_{\op} \Phi_s^+ \left[\left(\Phi_s^+ + b\right)^n - b^n\right] dx - \delta A_s b^n \int_{\op} \ifae\left(u,\Phi_s-\Phi_e\right) dx 
				\end{align}
				In a similar manner we also have, 
				\begin{equation}\label{petau}
					\tau \int_{\Omega} \Phi_e \left(\Phi_e^+ + b\right)^n dx = \tau \int_{\Omega} \Phi_e^+ \left[ \left(\Phi_e^+ + b\right)^n - b^n \right] dx + \delta A_s b^n \int_{\op} \ifae\left(u,\Phi_s-\Phi_e\right) dx 
				\end{equation}
				Then we substitute \eqref{pstau} into \eqref{linfone}, and \eqref{petau} into \eqref{linftwo}, and then add the two resulting equations together to derive, 
				\begin{align}\label{linfthree}
					n & \int_{\Omega} \sigma_e \left(\Phi_e^+ + b\right)^{n-1} \left|\nabla\Phi_e^+ \right|^2 dx + n \int_{\op} \sigma_s \left(\Phi_s^+ + b\right)^{n-1} \left|\nabla\Phi_s^+ \right|^2 dx \nonumber \\
					& \leq \delta d_1\int_{\Omega} \sigma_e\left(u_0-\te\left(u_0-u\right)\right) \left(\Phi_e^+ + b\right)^{n-1} {\mathbf{f}} \cdot \nabla \Phi_e^+ dx \nonumber \\
					& - \delta A_s \int_{\op} \ifae\left(u,\Phi_s-\Phi_e\right)\left[\left(\Phi_s^+ + b\right)^n - \left(\Phi_e^+ + b\right)^n\right] dx \nonumber \\
					& + \delta \int_{\partial\op} I \left(\Phi_s^+ + b\right)^n dS \nonumber \\
					& = I_1 + I_2 + I_3 
				\end{align}
				We now proceed to estimate each of $I_1$-$I_3$. For $I_1$ we have from (H5),
				\begin{align}\label{ione}
					I_1 & \leq cn\left(\left\|u_0\right\|_{\infty,\Omega} + \varepsilon\right)\int_{\Omega} \sigma_e \left(\Phi_e^+ + b\right)^{n-1}\left|\nabla \Phi_e^+ \right| dx \nonumber \\
					& \leq \frac{n}{2} \int_{\Omega} \sigma_e \left(\Phi_e^+ + b\right)^{n-1} \left|\nabla \Phi_e^+ \right|^2 dx \nonumber \\
					&+ cn\left(\left\|u_0\right\|_{\infty,\Omega} +\varepsilon\right)^2\int_{\Omega}\left(\Phi_e^+ + b\right)^{n-1} dx+
				\end{align}
				For $I_2$, we first use the fact that $\ifae(y_1,y_2)$ is increasing in its second variable to conclude that, 
				\begin{equation}
					G = \left[\ifae\left(u,\Phi_s-\Phi_e\right) - \ifae\left(u,0\right)\right]\left[\left(\Phi_s^+ + b\right)^n - \left(\Phi_e^+ + b\right)^n\right] \geq 0 \nonumber 
				\end{equation}
				Then, we have, 
				\begin{align}\label{itwo}
					I_2 &= -\delta A_s \int_{\op} G dx - \delta A_s \int_{\op} \ifae\left(u,0\right)\left[\left(\Phi_s^+ + b\right)^n - \left(\Phi_e^+ + b\right)^n \right] dx \nonumber \\
					& \leq - \delta  A_s \int_{\op} \ifae\left(u,0\right)\left[\left(\Phi_s^+ + b\right)^n - \left(\Phi_e^+ + b\right)^n \right] dx
				\end{align}
				Next, we use the Trace Theorem to estimate $I_3$, 
				\begin{align}\label{ithree}
					I_3 & \leq I_{max} \int_{\partial\op} \left(\Phi_s^+ + b\right)^n dS \nonumber \\
					& \leq c n I_{max} \int_{\op} \left(\Phi_s^+ + b\right)^{n-1} \left|\nabla \Phi_s^+ \right| dx + c I_{max} \int_{\op} \left(\Phi_s^+ + b\right)^n dx \nonumber \\
					& \leq \frac{n}{2} \int_{\op} \left(\Phi_s^+ + b\right)^{n-1} \left|\nabla\Phi_s^+ \right|^2 dx + c n I^2_{max} \int_{\op} \left(\Phi_s^+ + b\right)^{n-1} dx \nonumber \\
					& + c I_{max} \int_{\op} \left(\Phi_s^+ + b\right)^n dx 
				\end{align}
				We then substitute each of \eqref{ione}, \eqref{itwo}, and \eqref{ithree} into \eqref{linfthree} to derive, 
				\begin{align}\label{linffour}
					\frac{2n}{(n+1)^2} & \int_{\Omega} \sigma_e \left|\nabla\left(\Phi_e^+ + b \right)^{\frac{n+1}{2}}\right|^2 dx + \frac{2n}{(n+1)^2} \int_{\op} \sigma_s \left|\nabla\left(\Phi_s^+ + b \right)^{\frac{n+1}{2}}\right|^2 dx \nonumber \\
					& \leq cn\left(\left\|u_0\right\|_{\infty,\Omega} + \varepsilon\right)^2 \int_{\Omega} \left(\Phi_e^+ + b\right)^{n-1} dx \nonumber \\
					& + \delta  A_s \int_{\op} \ifae\left(u,0\right)\left[\left(\Phi_s^+ + b\right)^n - \left(\Phi_e^+ + b\right)^n \right] dx \nonumber \\
					& + cn I^2_{max} \int_{\op} \left(\Phi_s + b \right)^{n-1} dx + c I_{max} \int_{\op} \left(\Phi_s + b\right)^n dx 
				\end{align}
				Then, let 
				\begin{equation}
					W = \max\{\Phi_e^+, \Phi_s^+\chi_{\op}\} \nonumber 
				\end{equation}
				Upon using the Sobolev Embedding Theorem and \eqref{linffour}, we estimate, 
				\begin{align}\label{linffive}
					\left\|\left(W + b\right)^{n+1} \right\|_{\frac{N}{N-2},\Omega} & = \left( \int_{\Omega} \left( W + b\right)^{\frac{(n+1)}{2}\frac{2N}{N-2}} dx \right)^{\frac{N-2}{N}} \nonumber \\
					& \leq c \int_{\Omega} \left|\nabla \left(W+b\right)^{\frac{n+1}{2}}\right|^2 dx + c \int_{\Omega} \left(W+b\right)^{n+1} dx \nonumber \\
					& \leq c \int_{\Omega} \left|\nabla \left(\Phi_e^+ + b \right)^{\frac{n+1}{2}} \right|^2 dx + c \int_{\op} \left|\nabla\left(\Phi_s^+ + b\right)^{\frac{n+1}{2}} \right|^2 dx \nonumber \\
					& + c \int_{\Omega} \left( W + b\right)^{n+1} dx \nonumber \\
					& \leq c\left(n+1\right)^2\left( \left(\left\|u_0\right\|_{\infty,\Omega} + \varepsilon\right)^2 + I^2_{max}\right) \int_{\Omega} \left(W +b \right)^{n-1} dx \nonumber \\
					& + c (n+1)^2\int_{\op} \left|\ifae\left(u,0\right)\right|\left(W + b\right)^n dx \nonumber \\
					&+ c(n+1)^2 I_{max} \int_{\Omega} \left(W + b\right)^n dx + c \int_{\Omega} \left(W + b\right)^{n+1} dx \nonumber \\
					& \leq c(n+1)^2 \left( \left\|u_0\right\|_{infty,\Omega} + I_{max} + \varepsilon\right)^2 \left\|\left(W + b\right)^{n-1} \right\|_{\frac{p}{p-1},\Omega} \nonumber \\
					& + c(n+1)^2 \left( \left\|\ifae\left(u,0\right)\right\|_{p,\op} + I_{max} \right) \left\|\left(W + b\right)^{n} \right\|_{\frac{p}{p-1},\Omega} \nonumber \\
					& + c \left\|\left(W + b\right)^{n+1} \right\|_{\frac{p}{p-1},\Omega} \nonumber \\
					& \leq c(n+1)^2 \left\|\left(W + b\right)^{n+1} \right\|_{\frac{p}{p-1},\Omega} 
				\end{align}
				We then choose $p$ to be large enough that, 
				\begin{equation}
					\ell = \frac{N}{N-2}\frac{p-1}{p} > 1. \nonumber  
				\end{equation}
				Then, for each $i=2,3,...$ we put, 
				\begin{equation}
					n+1 = \ell^i. \nonumber 
				\end{equation}
				We take the $(n+1)^{st}$ root of each side of \eqref{linffive} to obtain, 
				\begin{align}
					\left\|W + b\right\|_{\frac{p \ell^{i+1}}{p-1}, \Omega} & \leq c^{\frac{1}{(n+1)}} \left(n+1\right)^{\frac{2}{n+1}}  \left\| W + b\right\|_{\frac{p\ell^i}{p-1}, \Omega} \nonumber \\
					& = c^{\frac{1}{\ell^i} } \ell^{\frac{2}{\ell^i} } \left\| W + b\right\|_{\frac{p\ell^i}{p-1},\Omega} \nonumber 
				\end{align}
				Iterating on $i$ then yields, 
				\begin{equation}
					\left\| W + b \right\|_{\frac{p\ell^{i+1}}{p-1}, \Omega} \leq c^{\sum_{j=1}^{i} \frac{1}{\ell^j} } \ell^{2\sum_{j=1}^{i} \frac{1}{\ell^{j}} } \left\| W + b \right\|_{\frac{p\ell}{p-1},\Omega} \nonumber 
				\end{equation}
				Then we let $i\rightarrow \infty$ and use the interpolation inequality to obtain, 
				\begin{align}
					\left\| W + b \right\|_{\infty,\Omega} &\leq c \left\| W + b \right\|_{\frac{p\ell}{p-1}, \Omega} \nonumber \\
					& \leq \sigma \left\| W + b \right\|_{\infty,\Omega} + c\left(\sigma\right) \left\| W + b\right\|_{1,\Omega}. \nonumber 
				\end{align}
				Upon choosing $\sigma$ to be sufficiently small we then have by using lemma \ref{potLtwobound}, 
				\begin{align}
					\left\| W \right\|_{\infty,\Omega} & \leq \left\| W + b \right\|_{\infty,\Omega} \nonumber \\
					& \leq c \left\| W + b \right\|_{1, \Omega} \nonumber \\
					& \leq c \left\| W \right\|_{1,\Omega} + cb \nonumber \\
					& \leq c \left\|\Phi_e\right\|_{2,\Omega} + c \left\| \Phi_s \right\|_{2,\op} + c \nonumber \\
					& \leq c. \nonumber 
				\end{align}
				Subsequently, we can conclude that, 
				\begin{equation}
					\sup_{\Omega} \Phi_e^+ + \sup_{\op} \Phi_s^+ \leq c \nonumber 
				\end{equation}
				By applying the same method to the negative parts of $\Phi_e$ and $\Phi_s$ we can then conclude the proof. 
			\end{proof}
			Finally, having proven \ref{LScond}, we can conclude by the Leray-Schauder Theorem that $\mathbb{B}$ has a fixed point. Clearly this fixed point is the unique solution to \eqref{psapprox}-\eqref{peapproxboundary}. The proof is complete. 
		\end{proof}
		
		Next, we will justify passing to the limit in equations, \eqref{psapprox}-\eqref{peapproxboundary}. This will yield the following lemma,  
		\begin{lemma}\label{potexists}
			If $u \in C\left(\overline{\ot}\right)$, then for any $t\in\left[0,T\right]$ there is a weak solution, $\left(\Phi_s, \Phi_e\right) \in \mathbb{X}_1 \cap \mathbb{Y}_1$ to the problem, 
			\begin{align}
				-\mdiv\left(\sigma_s\nabla\Phi_s\right)  &= -A_s\ifae\left(u, \Phi_s-\Phi_e\right), \ \ \mbox{ in $\op$, } \label{ps} \\
				-\mdiv\left(\sigma_e\nabla \Phi_e\right)  &= -\mdiv\left(d_1\sigma_e \left( u_0 - \te\left(u_0-u\right)\right) \mathbf{f} \right) + A_s \ifae\left( u, \Phi_e - \Phi_s\right)\chi_{\op}, \ \  \mbox{ in $\Omega$, } \label{pe} \\ 
				\sigma_s \nabla \Phi_s \cdot \mathbf{n} &= I, \ \ \mbox{ on $\partial\op$, } \label{psboundary} \\
				\sigma_e \nabla \Phi_e \cdot \mathbf{n} &= 0 , \ \ \mbox{ on $\partial\Omega$. } \label{peboundary} 
			\end{align}
			In addition we have, 
			\begin{equation}\label{potsumiszero}
				\int_{\Omega} \Phi_e dx + \int_{\op} \Phi_s dx = 0. 
			\end{equation}
		\end{lemma}
		
		\begin{proof}
			For each $\tau > 0$, we let $\left(\pst,\pet\right)$ be the solutions given by \ref{potapprox} to the Boundary-Value problem, 
			\begin{align}
				-\mdiv\left(\sigma_s\nabla \pst \right) + \tau \pst &= -A_s \ifae\left(u,\pst-\pet\right), \ \mbox{ in $\op$, } \label{tone} \\
				-\mdiv\left(\sigma_e \nabla \pet\right) + \tau \pet &= -\mdiv\left(\sigma_e d_1 \left(u_0-\te\left(u_0-u\right)\right) \mathbf{f} \right) \nonumber \\
				& + A_s \ifae\left(u,\pst-\pet\right)\chi_{\op}, \ \mbox{ in $\Omega$, } \label{ttwo} \\
				\sigma_s\nabla \pst \cdot \mathbf{n} &= I, \ \mbox{ on $\partial\op$, } \label{tthree} \\
				\sigma_e \nabla \pet \cdot \mathbf{n} &= 0, \ \mbox{ on $\partial\Omega$. } \label{tfour} 
			\end{align}
			Then, on account of lemma \ref{potlinfbound}, there is a constant which is independent of $\tau$ so that, 
			\begin{equation}
				\sup_{\Omega} \left|\pet\right| + \sup_{\op} \left| \pst \right| \leq c \nonumber 
			\end{equation}
			Subsequently, we can conclude that $\ifae\left(u, \pst-\pet\right)$ is uniformly bounded in $L^{\infty}\left(\op\right)$, independently from $\tau$. Using $\pst$ as a test function in \eqref{tone} and $\pet$ as a test function in \eqref{ttwo} and adding together the two resulting equations we then derive, 
			\begin{equation}
				\int_{\Omega} \left|\nabla \pet \right|^2 dx + \int_{\op} \left|\nabla \pst \right|^2 dx \leq c \nonumber 
			\end{equation}
			As a consequence we may pass to a subsequence, which we won't relabel, so that, 
			\begin{align}
				\pet &\rightarrow \Phi_e, \ \mbox{ weakly in $W^{1,2}\left(\Omega\right)$, and strongly in $L^p\left(\Omega\right)$ for all $1\leq p < \infty$, } \nonumber \\
				\pst &\rightarrow \Phi_s, \ \mbox{ weakly in $W^{1,2}\left(\op\right)$, and strongly in $L^p\left(\op\right)$ for all $1\leq p < \infty$. } \nonumber 
			\end{align}
			Then, since the function $\ifae\left(u,\pst-\pet\right)$ is uniformly bounded we can also conclude that, 
			\begin{equation}
				\ifae\left(u,\pst-\pet\right) \rightarrow \ifae\left(u,\Phi_s-\Phi_e\right) \ \mbox{ strongly in $L^p\left(\op\right)$ for all $1\leq p < \infty$. } \nonumber 
			\end{equation}
			We are now able to pass to the limit in equation \eqref{tone}-\eqref{tfour}. The proof is complete. 
		\end{proof}

		\section{Proof of the Main Theorem}
		
		We are now ready to prove our Main Theorem, \ref{MainTheorem}. To do so we will use a fixed-point mapping as before. 
		
		Let $L_0$ and $\varepsilon$ be given as in \eqref{epsiloncondition}, and $\ifae\left(u,\Phi_s-\Phi_e\right)$ as in \eqref{ifaraepsilon}. Then, using $\ifae$ we define,
		\begin{align}
			\qe\left(u, \Phi_s-\Phi_e, \nabla \Phi_s, \nabla \Phi_e\right) &= \left(A_s\ifae\left(u,\Phi_s-\Phi_e\right)(\Phi_s - \Phi_e)+\sigma_s|\nabla\Phi_s|^2\right)\chi_{\op}\nonumber\\
			&+\sigma_e|\nabla\Phi_e|^2+d_1\sigma_e \left(u_0-\te\left(u_0-u\right)\right) \bf{f}\cdot\nabla\Phi_e. \nonumber 
		\end{align}
		
		Then, for any $T>0$ we consider the problem, 
		\begin{align}
			\left(C_p\rho\right)\partial_t u -\mdiv\left(k\nabla u\right) &= \qe\left(u,\Phi_s-\Phi_e,\nabla\Phi_s,\nabla\Phi_e\right), \ \mbox{ in $\ot$, } \label{mainone} \\
			-\mdiv\left(\sigma_s\nabla\Phi_s\right) &= -A_s\ifae\left(u,\Phi_s-\Phi_e\right), \ \mbox{ in $\op_T$ } \label{maintwo} \\
			\mdiv\left(\sigma_e\nabla\Phi_e\right) &= - \mdiv\left(\sigma_e d_1\left(u_0-\te\left(u_0-u\right)\right){\mathbf{f}}\right) + A_s \ifae\left(u,\Phi_s-\Phi_e\right), \ \mbox{ in $\ot$ } \label{mainthree} 
		\end{align}
		
		We first prove the existence of solutions to \eqref{mainone}-\eqref{mainthree}. To do so we define a mapping $\mathbb{J}$ taking $C\left(\overline{\ot}\right)$ into itself in the following manner; given $v\in C\left(\ot\right)$ we first let $\left(\Phi_s,\Phi_e\right) \in \mathbb{X}_1 \cap \mathbb{Y}_1$ be the unique solutions given by lemma \ref{potexists} to the problem, 
		\begin{align}
			-\mdiv\left(\sigma_s\nabla\Phi_s\right)  &= -A_s\ifae\left(v, \Phi_s-\Phi_e\right), \ \ \mbox{ in $\op$, } \label{mtone} \\
			-\mdiv\left(\sigma_e\nabla \Phi_e\right)  &= -\mdiv\left(d_1\sigma_e \left( u_0 - \te\left(u_0-v\right)\right) \mathbf{f} \right) + A_s \ifae\left( v, \Phi_e - \Phi_s\right)\chi_{\op}, \ \  \mbox{ in $\Omega$, } \label{mttwo} \\ 
			\sigma_s \nabla \Phi_s \cdot \mathbf{n} &= I, \ \ \mbox{ on $\partial\op$, } \label{mtthree} \\
			\sigma_e \nabla \Phi_e \cdot \mathbf{n} &= 0 , \ \ \mbox{ on $\partial\Omega$. } \label{mtfour} 
		\end{align}
		By lemma \ref{potlinfbound} we can conclude that, 
		\begin{equation}
			\sup_{\Omega} \left|\Phi_e\right| + \sup_{\op} \left|\Phi_s\right| \leq c \nonumber 
		\end{equation}
		Where the constant $c$ depends not on $v$ but on the norm of $u_0$ and $\varepsilon$. Consequently we may conclude that $\ifae\left(v,\Phi_s-\Phi_e\right)$ is bounded in $L^{\infty}\left(\Omega\right)$. Let $p> N+2$. By lemma \ref{lpelliptic} we may then conclude that there is a constant c so that, 
		\begin{equation}
			\left\|\Phi_e\right\|_{W^{1,p}\left(\Omega\right)} + \left\|\Phi_s\right\|_{W^{1,p}\left(\op\right)} \leq c \nonumber 
		\end{equation}
		Raising each term in the above equation to the $p$th power, and integrating with respect to $t$ we then obtain, 
		\begin{equation}
			\int_{\ot} \left|\nabla \Phi_e\right|^p dx dt + \int_{\op_T} \left|\nabla\Phi_s\right|^p dxdt \leq cT. \nonumber 
		\end{equation}
		Consequently we can then conclude that, 
		\begin{equation}
			\qe\left(v,\Phi_s-\Phi_e, \nabla\Phi_s,\nabla\Phi_e\right) \in L^{\frac{p}{2}}\left(\ot\right) \nonumber 
		\end{equation}
		Next, we form the problem, 
		\begin{align}
			\left(\rho C_p\right)\partial_t u - \mdiv\left(k\nabla u \right) &= \qe\left(v,\Phi_s-\Phi_e, \nabla\Phi_s,\nabla\Phi_e\right), \ \mbox{ in $\ot$, } \label{uone} \\
			k\nabla u \cdot \mathbf{n} &= k_1(u_0 - \te\left(u_0 -v\right)-T_a), \ \mbox{on $\Sigma_T$,} \label{utwo} \\
			u(x,0) &= u_0(x), \ \mbox{ on $\Omega$ } \label{uthree}
		\end{align}
		Since, $\frac{p}{2} > \frac{N}{2} + 1$, we can conclude from the classical existence and regularity theory for linear parabolic equations (\cite{LSU}) that there is a unique $ u \in C^{\alpha,\frac{\alpha}{2}}\left(\overline{\ot}\right) \cap L^{\infty}\left(0,T;W^{1,2}\left(\Omega\right)\right)$ solving \eqref{uone}-\eqref{uthree}. We then set $\mathbb{J}\left(v\right) = u$. As solutions to \eqref{uone}-\eqref{uthree} are unique, we can conclude that $\mathbb{J}$ is well-defined. Since the space $C^{\alpha,\frac{\alpha}{2}}\left(\overline{\ot}\right)$ is compactly embedded in $C\left(\overline{\ot}\right)$ we can also conclude that $\mathbb{J}$ must map bounded sets into precompact ones. 
		
		\begin{lemma}\label{jcont}
			The mapping $\mathbb{J}$ is continuous on $C\left(\overline{\ot}\right)$. 
		\end{lemma}
		\begin{proof}
			Let $\{v_n\}$ be a sequence in $C\left(\overline{\ot}\right)$ and $v \in C\left(\overline{\ot}\right)$ be such that, 
			\begin{equation}\label{vconv}
				v_n \rightarrow v, \ \mbox{ uniformly on $\overline{\ot}$ } \nonumber 
			\end{equation}
			Then for each $n=1,2,3,...$ we first let, $\left(\psn,\pen\right)$ be the solutions to the problem,
			\begin{align}
				-\mdiv\left(\sigma_s\nabla\psn\right)  &= -A_s\ifae\left(v_n, \psn-\pen\right), \ \ \mbox{ in $\op$, } \label{mtfive} \\
				-\mdiv\left(\sigma_e\nabla \pen\right)  &= -\mdiv\left(d_1\sigma_e \left( u_0 - \te\left(u_0-v_n\right)\right) \mathbf{f} \right) + A_s \ifae\left( v_n, \psn - \pen\right)\chi_{\op}, \ \  \mbox{ in $\Omega$, } \label{mtsix} \\ 
				\sigma_s \nabla \psn \cdot \mathbf{n} &= I, \ \ \mbox{ on $\partial\op$, } \label{mtseven} \\
				\sigma_e \nabla \pen \cdot \mathbf{n} &= 0 , \ \ \mbox{ on $\partial\Omega$. } \label{mteight} 
			\end{align}
			As before, we can conclude that the sequences $\{\psn\}$ and $\{\pen\}$ satisfy, 
			\begin{align}
				\sup_{\ot} \left|\pen\right| &+ \sup_{\op_T} \left|\psn\right|  \leq cT \label{ptmtsup} \\
				\int_{\ot} \left|\nabla\pen\right|^p dxdt &+ \int_{\op_T} \left|\nabla\psn\right|^p dxdt \leq cT \label{ptmtlp} 
			\end{align}
			Where the constant $c$ is independent of $n$. Therefore we can conclude that $\{\qe\left(v_n,\psn-\pen,\nabla\psn,\nabla\pen\right)\}$ is uniformly bounded in $L^{\frac{p}{2}}\left(\ot\right)$. We then let $u_n$ be the solution to, 
			\begin{align}
				\left(\rho C_p\right)\partial_t u_n - \mdiv\left(k\nabla u_n \right) &= \qe\left(v_n,\psn-\pen, \nabla\psn,\nabla\pen\right), \ \mbox{ in $\ot$, } \label{unine} \\
				k\nabla u_n \cdot \mathbf{n} &= k_1(u_0 - \te\left(u_0 -v_n\right)-T_a), \ \mbox{on $\Sigma_T$,} \label{uten} \\
				u_n(x,0) &= u_0(x), \ \mbox{ on $\Omega$ } \label{ueleven}
			\end{align}
			Clearly, $\mathbb{J}\left(v_n\right) = u_n$ for each $n=1,2,3,...$. Then, since $\{\qe\left(v_n,\psn-\pen,\nabla\psn,\nabla\pen\right)\}$ is uniformly bounded in $L^{\frac{p}{2}}\left(\ot\right)$ we can conclude that the sequence $\{u_n\}$ is uniformly bounded in $C^{\alpha,\frac{\alpha}{2}}\left(\overline{\ot}\right) \cap L^{\infty}\left(0,T;W^{1,2}\left(\Omega\right)\right)$. Therefore, we may pass to subsequences, which we won't relabel, so that, 
			\begin{align}
				u_n &\rightarrow u, \mbox{ uniformly on $\overline{\ot}$, and weakly in $L^2\left(0,T;W^{1,2}\left(\ot\right)\right)$, } \nonumber \\
				\psn &\rightarrow \Phi_s, \mbox{ weakly in $\L^p\left(0,T;W^{1,p}\left(\op\right)\right)$, } \nonumber \\
				\pen &\rightarrow \Phi_e, \mbox{ weakly in $\L^p\left(0,T;W^{1,p}\left(\Omega\right)\right)$ } \nonumber 
			\end{align}
			In order to have any convergence for the nonlinear terms $\ifae\left(v_n,\psn-\pen\right)$ and $\qe$ we need to upgrade the convergence of the sequence $\{\left(\psn,\pen\right)\}$ from weak to strong convergence. The next claim deals with this, 
			\begin{clm}\label{strongconv}
				The sequence $\{\psn\}$ is precompact in $L^2\left(0,T; W^{1,2}\left(\op\right)\right)$, and $\{\pen\}$ is precompact in $L^2\left(0,T; W^{1,2}\left(\op\right)\right)$.
			\end{clm}
			\begin{proof}
				We first use \eqref{mtfive} and \eqref{mtsix} to derive, 
				\begin{align}
					- \mdiv\left(\sigma_s \nabla\left(\Phi_s^{(n_1)}-\Phi_s^{(n_2)}\right)\right) &= -A_s \ifae\left(v_{n_1},\Phi_s^{(n_1)} - \Phi_e^{(n_1)}\right) \nonumber \\
					& + A_s \ifae\left(v_{n_2},\Phi_s^{(n_2)}-\Phi_s^{(n_2)}\right) \ \mbox{ in $\op_T$, } \label{cone}
				\end{align}
				and, 
				\begin{align}
					-\mdiv\left(\sigma_e\nabla\left(\Phi_e^{(n_1)} - \Phi_e^{(n_2)}\right)\right) &= A_s \ifae\left(v_{n_1},\Phi_s^{(n_1)} - \Phi_e^{(n_1)}\right)\chi_{\op} \nonumber \\ 
					& - A_s \ifae\left(v_{n_2},\Phi_s^{(n_2)}-\Phi_s^{(n_2)}\right)\chi_{\op} \ \mbox{ in $\Omega_T$, } \label{ctwo}
				\end{align}
				
				Using $\left(\Phi_s^{(n_1)}-\Phi_s^{(n_2)}\right)$ as a test function in \eqref{cone} and $\left(\Phi_e^{(n_1)}-\Phi_e^{(n_2)}\right)$ as a test function in \eqref{ctwo} and adding the resulting two equations together then yields, 
				\begin{align}\label{cmtone}
					& \int_{\Omega} \sigma_e \left|\nabla\left(\Phi_e^{(n_1)}-\Phi_e^{(n_2)}\right)\right|^2 dx + \int_{\op} \sigma_s\left|\nabla\left(\Phi_s^{(n_1)}-\Phi_s^{(n_2)}\right)\right|^2 dx \nonumber \\
					& + A_s \int_{\op} G dx = 0
				\end{align}
				Where,
				\begin{equation}
					G = \left[\ifae\left(v_{n_1},\Phi_s^{(n_1)}-\Phi_e^{(n_2)}\right) -\ifae\left(v_{n_2},\Phi_s^{(n_2)}-\Phi_e^{(n_2)}\right)\right]\left[\left(\Phi_s^{(n_1)}-\Phi_e^{(n_1)}\right)-\left(\Phi_s^{(n_2)}-\Phi_s^{(n_2)}\right)\right] \nonumber
				\end{equation}
				We then use the Mean value theorem to find, 
				\begin{align}
					\ifae&\left(v_{n_1},\Phi_s^{(n_1)}-\Phi_e^{(n_1)}\right) - \ifae\left(v_{n_2},\Phi_s^{(n_2)}-\Phi_e^{(n_2)}\right) \nonumber \\
					&= \ifae\left(v_{n_1},\Phi_s^{(n_1)}-\Phi_e^{(n_1)}\right) - \ifae\left(v_{n_2},\Phi_s^{(n_1)}-\Phi_e^{(n_1)}\right) \nonumber \\
					&+ \ifae\left(v_{n_2},\Phi_s^{(n_1)}-\Phi_e^{(n_1)}\right) - \ifae\left(v_{n_2},\Phi_s^{(n_2)}-\Phi_e^{(n_2)}\right) \nonumber \\
					& = \partial_{y_1}\ifae\left(\xi,\Phi_s^{(n_1)}-\Phi_e^{(n_1)}\right) \left(v_{n_1}-v_{n_2}\right) \nonumber \\
					& + \partial_{y_2} \ifae\left(v_{n_2},\eta\right)\left[\left(\Phi_s^{(n_1)}-\Phi_e^{(n_1)}\right)-\left(\Phi_s^{(n_2)}-\Phi_e^{(n_2)}\right)\right] \nonumber 
				\end{align}
				Using this in $G$ and substituting into \eqref{cmtone}, we derive from \eqref{ptmtsup} and \eqref{partialy2} that, 
				\begin{align}
					\int_{\Omega} & \sigma_e \left|\nabla\left(\Phi_e^{(n_1)}-\Phi_e^{(n_2)}\right)\right|^2 dx + \int_{\op} \sigma_s \left|\nabla\left(\Phi_s^{(n_1)}-\Phi_s^{(n_2)}\right)\right|^2 dx \nonumber \\
					& + A_s c_0 \int_{\op} \left[\left(\Phi_s^{(n_1)}-\Phi_e^{(n_1)}\right)-\left(\Phi_s^{(n_2)}-\Phi_e^{(n_2)}\right)\right]^2 dx \nonumber \\
					& \leq c \left\| v_{n_1} - v_{n_2} \right\|_{\infty,\Omega} \nonumber 
				\end{align}
				In light of \eqref{potsumiszero} we have, 
				\begin{equation}
					\int_{\Omega} \left( \Phi_e^{(n_1)}-\Phi_e^{(n_2)} \right) dx + \int_{\op} \left( \Phi_s^{(n_1)} -\Phi_s^{(n_2)} \right) dx = 0 \nonumber 
				\end{equation}
				We are now in a position to repeat the argument in \eqref{PI}, from which we derive, 
				\begin{align}
					& \int_{\Omega} \left[\Phi_e^{(n_1)}-\Phi_e^{(n_2)}\right]^2 dx + \int_{\op} \left[\Phi_s^{(n_1)}-\Phi_s^{(n_2)}\right]^2 dx \nonumber \\
					& + \int_{\Omega} \sigma_e \left|\nabla\left(\Phi_e^{(n_1)}-\Phi_e^{(n_2)}\right)\right|^2 dx + \int_{\op} \sigma_s \left|\nabla\left(\Phi_s^{(n_1)}-\Phi_s^{(n_2)}\right)\right|^2 dx \nonumber \\
					& \leq c \left\| v_{n_1} - v_{n_2} \right\|_{\infty,\Omega} \nonumber 
				\end{align}
				Integrating this inequality with respect to $t$ we then derive that $\{\psn\}$ is Cauchy in $L^2\left(0,T;W^{1,2}\left(\op\right)\right)$ and $\{\pen\}$ is Cauchy in $L^2\left(0,T;W^{1,2}\left(\Omega\right)\right)$. This concludes the proof of claim \ref{strongconv}. 
			\end{proof}
			Continuing the proof of lemma \ref{jcont}, we can conclude from  claim \ref{strongconv} and \eqref{ptmtsup} that, 
			\begin{align}
				\Phi_s^{(n)} &\rightarrow \Phi_s, \ \mbox{ strongly in $L^2\left(0,T;W^{1,2}\left(\op\right)\right)$, and strongly in $L^q\left(\op_T\right) \ \forall 1\leq q < \infty$, } \nonumber \\
				\Phi_s^{(n)} &\rightarrow \Phi_s, \ \mbox{ strongly in $L^2\left(0,T;W^{1,2}\left(\op\right)\right)$, and strongly in $L^q\left(\op_T\right) \ \forall 1\leq q < \infty$ } \nonumber 
			\end{align}
			As a consequence, we can conclude that, 
			\begin{equation}
				\ifae\left(v_n,\Phi_s^{(n)} -\Phi_e^{n}\right) \rightarrow \ifae\left(v,\Phi_s-\Phi_e\right), \ \mbox{ strongly in $L^q\left(\op_T\right) \ \forall 1 \leq q < \infty$ } \nonumber 
			\end{equation}
			Then, due to \eqref{ptmtlp} we can conclude, 
			\begin{align}
				\nabla\Phi_s^{(n)} &\rightarrow \nabla \Phi_s \ \mbox{ strongly in $L^q\left(\op_T\right) \ \forall 1\leq q < p$ } \nonumber \\
				\nabla \Phi_e^{(n)} &\rightarrow \nabla \Phi_e \ \mbox{ strongly in $L^q\left(\op_T\right) \ \forall 1\leq q < p$ } \nonumber 
			\end{align}
			Now, we are able to conclude that, 
			\begin{align}
				\qe&\left(v_n,\psn-\pen, \nabla\psn,\nabla\pen\right) \rightarrow \qe\left(v,\Phi_s-\Phi_e, \nabla\Phi_s, \nabla \Phi_e \right) \nonumber \\
				& \mbox{ strongly in $L^q\left(\ot\right)$ for all $1 \leq q < \frac{p}{2}$. } \nonumber 
			\end{align}
			At this point we are now in a position to pass to the limit as $n\rightarrow \infty$ in \eqref{mtfive}-\eqref{mteight} and in \eqref{unine}-\eqref{ueleven} to find that $u = \mathbb{J}\left(v\right)$. Then, from the uniqueness of solutions to \eqref{mtone}-\eqref{mtfour} and \eqref{uone}-\eqref{uthree}, we can conclude that every subsequence of $\{u_n\}$ has a further convergent subsequence, each of which converges to $\mathbb{J}\left(v\right)$. As a result, the whole sequence $\{u_n\}$ converges to $\mathbb{J}\left(v\right)$. The proof of the continuity of $\mathbb{J}$ is complete. 
		\end{proof}
		
		The final piece in order to use the Leray-Schauder Theorem is to demonstrate that for any $\delta \in \left(0,1\right)$ and $u\in C\left(\overline{\ot}\right)$ so that, 
		\begin{equation}
			u = \delta\mathbb{J}\left(u\right) \label{deltaJone} 
		\end{equation}
		there is a constant $c$ so that, 
		\begin{equation}
			\left\|u \right\|_{\infty,\ot} \leq c \label{finalLScond}
		\end{equation}
		Equation \eqref{deltaJone} is equivalent to the existence of a triple $\left(u,\Phi_s,\Phi_e\right)$ satisfying the equations, 
		\begin{align}
			\left(C_p\rho\right)\partial_t u -\mdiv\left(k\nabla u\right) &= \delta \qe\left(u,\Phi_s-\Phi_e,\nabla\Phi_s,\nabla\Phi_e\right), \ \mbox{ in $\ot$, } \nonumber   \\
			-\mdiv\left(\sigma_s\nabla\Phi_s\right) &= -\delta A_s\ifae\left(u,\Phi_s-\Phi_e\right), \ \mbox{ in $\op_T$ } \nonumber  \\
			\mdiv\left(\sigma_e\nabla\Phi_e\right) &= - \delta\mdiv\left(\sigma_e d_1\left(u_0-\te\left(u_0-u\right)\right){\mathbf{f}}\right) + \delta A_s \ifae\left(u,\Phi_s-\Phi_e\right), \ \mbox{ in $\ot$ } \nonumber  \\
			-k\nabla u \cdot \mathbf{n} &= \delta k_1\left(u_0-\te\left(u_0-u\right) -T_a\right), \ \mbox{ on $\Sigma_T$, } \nonumber  \\
			\sigma_s\nabla\Phi_s \cdot \mathbf{n} &= \delta I, \ \mbox{ on $\partial\op\times\left(0,T\right)$ } \nonumber  \\
			\sigma_e \nabla \Phi_e \cdot \mathbf{n} &=0, \ \mbox{ on $\Sigma_T$ } \nonumber 
		\end{align}
		Then from lemma \ref{potlinfbound} we can conclude that there is a constant so that, 
		\begin{equation}
			\sup_{\ot}\left|\Phi_e\right| + \sup_{\op_T}\left|\Phi_s\right| \leq cT. \nonumber 
		\end{equation}
		Consequently we can conclude that $\ifae\left(u,\Phi_s-\Phi_e\right)$ is bounded in $L^{\infty}\left(\ot\right)$. As a result, by lemma \ref{lpelliptic} we have, 
		\begin{equation}
			\int_{\ot} \left|\nabla \Phi_e\right|^p dxdt + \int_{\op_T} \left|\nabla\Phi_s\right|^p dx dt \leq cT \nonumber 
		\end{equation}
		Then we also have that $\qe\left(u, \Phi_s-\Phi_e, \nabla\Phi_s, \nabla\Phi_e\right)$ is bounded in $L^{\frac{p}{2}}\left(\ot\right)$. By lemma \ref{linfparabolic} we then have equation \eqref{finalLScond}. We can now use the Leray-Schauder Theorem to conclude that $\mathbb{J}$ has a fixed-point. This gives us a weak solution to \eqref{mainone}-\eqref{mainthree} in the sense of \ref{weaksolutiondefn}. 
		
		Since $u \in C\left(\overline{\ot}\right)$ there exists a number $T^*>0$ so that, 
		\begin{equation}
			\left|u_0(x)-u\left(x,t\right)\right| \leq \varepsilon \ \mbox{ for all $t\in\left[0,T^*\right]$. } \nonumber 
		\end{equation}
		As a result we have,
		\begin{equation}
			u_0-\te\left(u_0-u\right) = u_0-u_0+u = u \ \mbox{ on $\Omega_{T^*}$. } \nonumber 
		\end{equation}
		Thus equations \eqref{temp}-\eqref{electrolytepot} are satisfied in $\Omega_{T^*}$. The proof of our Main Theorem \ref{MainTheorem} is now complete.

			\end{document}